\documentclass{amsart}
\usepackage{amssymb}
\usepackage{amsfonts}
\usepackage{amsbsy}
\usepackage{amsmath}
\usepackage{amsthm}

\usepackage{latexsym}
\usepackage[mathscr]{eucal}
\usepackage{verbatim}
\usepackage[normalem]{ulem}

\theoremstyle{plain}
\newtheorem{theorem}{Theorem}[section]
\newtheorem{lemma}[theorem]{Lemma}
\newtheorem{proposition}[theorem]{Proposition}
\newtheorem{definition}[theorem]{Definition}
\newtheorem{corollary}[theorem]{Corollary}

\newtheorem{remark}[theorem]{Remark}

\newtheorem{example}[theorem]{Example}

\newcommand{\A}{\mathcal{A}}

\newcommand{\be}{\begin{equation}\label}
\newcommand{\ee}{\end{equation}}
\newcommand{\bq}{\begin{equation*}}
\newcommand{\eq}{\end{equation*}}
\newcommand{\ba}{\begin{align*}}
\newcommand{\ea}{\end{align*}}
\newcommand{\bp}{\begin{proof}}
\newcommand{\ep}{\end{proof}}
\newcommand{\bL}{\begin{lemma}\label}
\newcommand{\eL}{\end{lemma}}
\newcommand{\bP}{\begin{proposition}\label}
\newcommand{\eP}{\end{proposition}}
\newcommand{\bC}{\begin{corollary}\label}
\newcommand{\eC}{\end{corollary}}
\newcommand{\bT}{\begin{theorem}\label}
\newcommand{\eT}{\end{theorem}}
\newcommand{\bR}{\begin{remark}\label}
\newcommand{\eR}{\end{remark}}
\newcommand{\bD}{\begin{definition}\label}
\newcommand{\eD}{\end{definition}}

\def\sideremark#1{\ifvmode\leavevmode\fi\vadjust{\vbox to0pt{\vss
\hbox to 0pt{\hskip\hsize\hskip1em
\vbox{\hsize2cm\tiny\raggedright\pretolerance10000
\noindent#1\hfill}\hss}\vbox to8pt{\vfil}\vss}}}

\numberwithin{equation}{section}

\thanks{This paper constitutes part of the doctoral dissertation of Catalin Dragan. This work was partially supported by the Simons Foundation grant No 245660 to Victor Kaftal}

\author{Catalin Dragan}
\address{Department of Mathematics\\
University of Cincinnati\\
P. O. Box 210025\\
Cincinnati, OH\\
45221-0025\\
USA}
\email{dragancn@mail.uc.edu}

\author{Victor Kaftal}

\address{Department of Mathematics\\
University of Cincinnati\\
P. O. Box 210025\\
Cincinnati, OH\\
45221-0025\\
USA}

\email{victor.kaftal@uc.edu}


\begin{document}
\title[Sums of equivalent sequences]{Sums of equivalent sequences of positive operators in von Neumann factors}

\maketitle

\begin{abstract}
Let $A$ be a positive operator in an infinite $\sigma$-finite von Neumann factor $\mathcal M$ and let $\{B_j\}_{j=1}^\infty \subset\mathcal M^+$. We give sufficient conditions for the decomposition $A=\sum_{j=1}^\infty C_j$ to hold when $C_j \sim B_j$ for all $j$ (the equivalence $C\sim B$ means $C=XX^*$ and $B=X^*X$ for some $X\in \mathcal M$) and when   $C_j$ are unitarily equivalent to $B_j$ for all $j$. This extends the work of Bourin and Lee in \cite{SMV}, \cite{UE} for the case of $B_j= B$ and $\mathcal {M}=B(\mathcal{H})$ and answers affirmatively their conjecture. For the case when $B_j= B$ for all $j$ we provide  necessary conditions, which in the type III case are also sufficient.
\end{abstract}

\section{Introduction}
In 1969 Fillmore characterized the (positive) finite rank operators that are sums of projections (\cite{Fp69}). In 1994 Wu and Choi announced that positive operators with essential norm strictly larger than 1 are  sums of projections (\cite{Wpy94} and \cite{ChoiWu2014}). See also \cite{KRS02} and  \cite{KRS03} for the special case of scalar multiples of the identity and \cite {DFKLOW}  and then \cite {AMRS} for a different approach motivated by frame theory.

The complete characterization of the positive operators that are infinite sums of projections converging in the strong topology, was obtained by Kaftal, Ng, and Zhang in \cite[Theorem 1.1]{SSP}. Their method did apply also to $\sigma$-finite von Neumann factors and using other methods their results were partially extended to some C*-algebras and their multiplier algebras in a series of articles  (\cite {KNZPISpan}-\cite {KNZCompJOT}. In particular, they obtained that  for $A\in \mathcal{M}^+$ to be a sum of projections (converging in the SOT) it is sufficient that 
\begin{itemize}
\item $\tau\big((A-I)_+\big)=\infty$, ($\tau$ a faithful normal semifinite trace)  when $\mathcal{M}$ is a type I$_\infty$ or a type II$_\infty$ factor;
\item $\|A\|> 1$, when $\mathcal{M}$ is a type III factor.
\end{itemize}
Notice that $$\|A\|_e>1 ~ \Rightarrow ~ \tau\big((A-I)_+\big)=\infty~\Rightarrow ~\|A\|_e\ge1,$$ where $\|\cdot\|_e$ is the essential norm relative to the ideal $\mathcal J$ of compact operators of $\mathcal{M}$ (the norm closed ideal generated by the finite projections of $\mathcal{M}$, also called  the Breuer ideal  \cite{Bm68}, \cite{Bm69}, see also \cite {Sm71}, \cite{Kv77} among others.)  Denote also by $\sigma_e(\cdot)$ the essential spectrum relative to this ideal, i.e., the spectrum of the canonical image in the (generalized) Calkin algebra $\mathcal M/\mathcal J$.

In $B(\mathcal{H})$, sums of projections can be further decomposed into sums of rank-one projections, which are all Murray-von Neumann equivalent. Thus an extension of \cite {SSP} is the study by
Bourin and Lee in \cite  {SMV} and  \cite  {UE} of decompositions of   positive operators $A\in B(\mathcal{H})$ into sums of  positive operators {\it equivalent} to a given positive operator $B\ne 0$.  

The first notion of equivalence they considered is the {\it Murray-von Neumann equivalence} (also called the {\it Pedersen equivalence} see \cite[Definition 6.1.2] {Bla88} and for more background  \cite {OrtegaRordamThiel}): $B\sim C$ if $B=XX^*$ and $C=X^*X$ for some $X$ or equivalently $B=VCV^*$ for some partial isometry $V$ such that $V^*V=R_C$, $VV^*= R_B$ where $R_B$, $R_C$ denote the range projections of $B$ and $C$ respectively (in a von Neumann algebra $\mathcal{M}$ we require that $X$, $V\in \mathcal{M}$). In this paper we will refer to $\sim$ simply as equivalence and we will denote by $\cong$ the relation of unitary equivalence.

Bourin and Lee  proved the following results: 
\begin{theorem}\label{SVM}
\cite [Theorem 1.2] {SMV} 
If $A, B\in B(\mathcal{H})^+$, $\tau((A-I)_+ )=\infty$, and $0\ne B$ is a  contraction, then $A=\sum_{j=1}^\infty B_j$ for some $B_j\sim B$. 
\end{theorem}

\begin{theorem}\label{UEBL}
\cite [Theorem 1.1] {UE} 
If $A, B\in B(\mathcal{H})^+$, $N_A=0$ ($N_A$ is the projection on the kernel of $A$),  $\|A\|_e\ge 1$, $B$ is a strict contraction and $0\in \sigma_e(B)$, then $A=\sum_{j=1}^\infty B_j$ for some  $B_j\cong B$.
\end{theorem}

By   {\it strict contraction}  Bourin and Lee  mean that $\|B\xi\|<\|\xi\|$ for every non-zero vector $\xi$.  If $B$ is a positive contractions this is equivalent to  the condition $\chi_{\{1\}}(B)=0$.  Indeed $\|B\xi\|=\|\xi\|$ if and only if $(I-B^2)\xi=0$, that is $ \xi\in \chi_{\{1\}}(B).$

They also conjectured in \cite {SMV} and \cite {UE}  that the same theorems hold for von Neumann factors.
The goal of the present paper is to prove their conjecture. 

Bourin and Lee based their proofs mostly on Bourin's ``pinching theorem"  \cite {PT} from which they obtained:
\begin{lemma}\label{Lemma 2.4}\cite[Lemma 2.4]{SMV}
If $A\in B(\mathcal{H})^+$, $\|A\|_e>1$, $\{B_j\}_{j=1}^\infty$ is a family of non-zero positive contractions, $\beta>0$, and $B_j\ge \beta R_{B_j}$ for all $j$ then $A= \sum_{j=1}^\infty C_j$ for some $C_j\sim B_j$.
\end{lemma}
One part of their proof depended on the previous result \cite[Theorem 1.1]{SSP} on decompositions into sums of projections,  but they announced an independent proof of that result. 

Our approach is the opposite.  We use decompositions into sums of projections from \cite{SSP}, which, as mentioned above, hold for $\sigma$-finite factors. From these we  obtain decompositions into sums of positive operators (Theorem \ref {equiv infinite trace}). As a consequence we answer affirmatively the conjecture of Bourin and Lee and then we proceed to deduce a form of the ``pinching theorem"  (Corollary  \ref{pinching corollary}). 

Our paper is organized as follows.

In section 2 we obtain some decomposition results for positive operators in a $\sigma$-finite factor $\mathcal M$ and we strengthen some results on decompositions into sums of projections for the type II$_\infty$ case.

In section 3 we use a series of reductions to sums of projections to obtain our main result, Theorem \ref {equiv infinite trace},  which proves  that an $A\in \mathcal{M}^+$ with $\tau((A-I)_+)=\infty$ can be decomposed as a sum $A=\sum_{j=1}^\infty C_j$ with $C_j\sim B_j$ for a preset sequence of contractions $\{B_j\}_{j=1}^\infty$ provided there is a $\beta >0$ and a non-zero projection $P$ such that $P\prec \chi_{(\beta,\infty)}(B_{j})$ for infinitely many indices $j$.

Of course that condition is always satisfied in the case of a single operator $B\ne 0$,  i.e., when $B_j=B$ for all $j$, which answers affirmatively the conjecture of Bourin and Lee. 

The connection between decompositions into sums of projections of a positive $A\in \mathcal M$ and the (block) diagonals of $A$  used in \cite{SSP}, \cite{FSP} is easily extended in Proposition \ref {block diagonal decomp} to sums of positive elements. Thus Corollary  \ref{pinching corollary}  provides a form of the  ``pinching theorem"  of  \cite {PT} for von Neumann factors. 

In section 4  we find sufficient conditions for the decomposition to hold with $C_j$ unitarily equivalent to $B_j$ (Theorem \ref {unit equiv sequence}) under the additional hypothesis that  $0\in \sigma_e(B_j)$ (resp., $0\in \sigma(B_j)$ if $\mathcal M$ is type III). Part of the proof of Theorem \ref {unit equiv sequence} is an adaptation to von Neumann factors of the methods used by Bourin and Lee in \cite{UE} for a single operator in $B(\mathcal{H})$. 

In section 5 we specialize the previous results to the case where $B_j=B$ for all $j$ and we find necessary conditions for that case (Proposition \ref{nec equiv}).  Theorem \ref {equiv} provides additional and independent sufficient conditions. These are also  sufficient when $\mathcal M$ is type III,  or when $\|A\|=\|A\|_e$ (Corollaries \ref {norm=essnorm} and \ref{N&S type III}). Similar results are obtained for unitary equivalence (Corollaries \ref {UE} and \ref{P:NecSuf}).

\section{Preliminary decompositions}  

Throughout this paper, unless otherwise stated, $\mathcal{M}$ denotes a $\sigma$-finite, infinite factor acting on a Hilbert space $\mathcal{H}$ and all the operators considered will belong to $\mathcal{M}$. When $\mathcal{M}$ is type I$_{\infty}$, $\tau$ denotes the standard trace, normalized on rank one projections, when $\mathcal{M}$ is type II$_{\infty}$, $\tau$ denotes a faithful, normal, semifinite, tracial weight. 

If $A\in \mathcal{M}$ we denote by $R_A$ and $N_A$  the range projection and the projection on the null space of $A$ respectively,  and when $A=A^*$, we denote by $A_+$ the positive part of $A$.

$\mathcal{J}$ denotes the norm-closed two-sided ideal generated by the finite projections of $\mathcal{M}$.  In particular, when $\mathcal{M}=B(\mathcal{H})$ then $\mathcal{J}= K(\mathcal{H})$, the ideal of compact operators,  and when $\mathcal{M}$ is type III, $\mathcal{J}=\{0\}$.  
When $\mathcal{M}$ is semifinite, denote by $\pi$ the canonical map from $\mathcal{M}$ onto $\mathcal{M}/\mathcal{J}$, which is a unital, surjective, *-homomorphism. Then the essential norm and essential spectrum of $A$ are defined as $\|A\|_e=\|\pi(A)\|$ and respectively $\sigma_e(A)=\sigma(\pi(A))$.

If $P$ is a non-zero projection in $\mathcal{M}$ we denote by $\mathcal{M}_P$ the compression of $\mathcal{M}$ by $P$ and by $\tau_P$ the restriction of $\tau$ on $\mathcal{M}_P$, i.e. $\mathcal{M}_P=P\mathcal{M}P|_{P\mathcal{H}}$ and $\tau_P(X_P)=\tau(PXP)$ for all $X\in\mathcal{M}^+$.

Notice that when $P$ is infinite and $AP=PA$ then $\|AP\|_e=\|A_P\|_e$ while $$\sigma_e(AP)=
\begin{cases}
\sigma_e(A_P), & \text{ if $I-P$ is finite},\\
\sigma_e(A_P)\cup\{0\}, & \text{ if $I-P$ is infinite},
\end{cases}
$$ where $\|A_P\|_e$ and $\sigma_e(A_P)$ are considered relative to $\mathcal{M}_P$ and $\tau_P$. 

We also make the convention that whenever we write $A=\bigoplus_{i=1}^{N} A_i$, where $N$ is some positive integer or $\infty$, then $A_i=AE_i$ for some mutually orthogonal projections $E_i\in \mathcal{M}$  that commute with $A$. Similarly, if we write $A=\bigoplus_{i=1}^{N} AE_i$ we will mean that the projections $E_i$ are mutually orthogonal and commute with $A$.

We collect below two known properties of the essential norm and essential spectrum. 
For $A\in \mathcal{M}^+$ and $\mathcal{M}$ semifinite we have that:

1) $\sigma_e(A)=\{\,\lambda\in\mathbb{R}\mid \tau\left(\chi_{(\lambda-\epsilon,\lambda+\epsilon)}(A)\right)=\infty,\forall\, \epsilon>0\,\},$ 

2) $\|A\|_e>\lambda$ if and only if  there exists $\epsilon>0$ such that $\tau\left(\chi_{(\lambda+\epsilon,\infty)}(A)\right)=\infty.$

We will also need the following decompositions of positive operators.\\

\begin{lemma} \label{gen ess spect}
Let $\mathcal{M}$ be semifinite, $A \in \mathcal{M}^+$, and $N\geq 1$ an integer or $N=\infty$. 

\item [ i)] 
If $\lambda \in \sigma_e(A)$ then $A=\bigoplus_{i=1}^{N}AE_i$ for some mutually orthogonal infinite projections $E_i\in \{A\}'$ with $\sum_{i=1}^N E_i=I$ and with $\lambda \in \sigma_e(A_{E_i})$ for every $i$.
\item [ ii)] $A=\bigoplus_{i=1}^{N}A_i$ for some $A_i\in \mathcal M^+$ with $\|A_i\|_e=\|A\|_e$ for every $i$. 
\item [ ii$\,'$)] If $N=\infty$ the positive operators $A_i$ in ii) can be chosen to be locally invertible, that is,  $A_i\ge \alpha_iR_{A_i}$ for some $\alpha_i>0$ (we use the convention that the operator 0 is locally invertible).
\item [ iii)] If  $\tau((A-tI)_+)=\infty$ for some $t>0$ then $A=\bigoplus_{i=1}^{N}A_i$ for some $A_i\in \mathcal M^+$ with $\tau\left((A_i-tI)_+\right)=\infty$ for all $i$.
\item [ iii$\,'$)] If $N=\infty$ the positive operators $A_i$ in iii) can be chosen to be locally invertible.\end{lemma} 
\begin{proof}
\item [ i)]  When $\tau\left(\chi_{\{\lambda\}}(A)\right)=\infty$, decompose $\chi_{\{\lambda\}}(A)=\sum_{k=1}^NQ_k$ into a sum of mutually orthogonal infinite projections $Q_k$. Take $E_k=Q_k$ for $k\ge2$ and $E_1=I-\sum_{k=2}^NQ_k$. It is immediate to verify that these projections satisfy the condition in i).

Assume therefore  that $\tau\left(\chi_{\{\lambda\}}(A)\right)<\infty$. Choose sequences  $\{c_k\}_{k=1}^{\infty}$ and $\{d_k\}_{k=1}^{\infty}$ such that $\lim_{k\to \infty}c_k=\lim_{k\to \infty}d_k=\lambda$, the intervals $[c_k,d_k)$ are disjoint and such that for every $L\ge 1$, $ \bigcup_{k=L}^{\infty}[c_k,d_k)= [\alpha_L, \beta_L)\setminus\{\lambda\}$ for some $\alpha_L< \lambda< \beta_L$. For instance one can take a strictly increasing sequence $\{a_k\}$ and a strictly decreasing sequence $\{b_k\}$, both convergent to $\lambda$, and then relabel the intervals $\{[a_k,a_{k+1})\}$ and $\{[b_{k+1},b_k)\}$ as $\{[c_k,d_k)\}$.  

 By 1) above,
$$
\tau\left(\chi_{\{\lambda\}}(A)\right)+\sum_{k=L}^{\infty}\tau\left(\chi_{[c_k,d_k)}(A)\right)=  \tau\left(\chi_{[\alpha_L,\beta_L)}(A)\right)= \infty.
$$
Hence $\sum_{k=L}^{\infty}\tau\left(\chi_{[c_k,d_k)}(A)\right)=\infty$ for every $L\geq 1$. Now split the sequence of intervals $\{[c_k,d_k)\}_{k=1}^{\infty}$ into $N$ disjoint subsequences $\{[c_k^i,d_k^i)\}_{k=1}^{\infty}$, $1\leq i\leq N$, such that $\sum_{k=L}^{\infty}\tau\left(\chi_{[c_k^i,d_k^i)}(A)\right)=\infty$ for every $L\geq 1$ and every $1\leq i\leq N$.

Let $\Delta_i=\bigcup_{k=1}^{\infty}[c_k^i,d_k^i)$ and $E_i=\chi_{\Delta_i}(A)$. Furthermore, replace the projection $E_1$ with $I-\sum_{k=2}^N E_k$. From the construction and property 1) of the essential spectrum it is easy to verify that the projections $\{E_k\}_{k=1}^N$ have the desired properties.

\item [ ii)] follows from i) by noticing that $A\geq 0$ implies $\|A\|_e \in \sigma_e(A)$.
\item [ iii)]  Because $\tau((A-tI)_+)=\infty$ we necessarily have $\tau\left(\chi_{(t,\infty)}(A)\right)=\infty$ which implies by property 2) of the essential norm that $\|A\|_e\geq t$. 

If $\|A\|_e>t$ then by ii) we have $A=\bigoplus_{i=1}^{N}A_i$ for some $\|A_i\|_e=\|A\|_e>t$ for every $i$. We are done because $\|A_i\|_e>t$ implies $\tau((A_i-tI)_+)=\infty$. 

Hence, assume that $\|A\|_e=t$. Then $\tau\left(\chi_{(s,\infty)}(A)\right)<\infty$ for every $s>t$, and in particular $\tau\left(\chi_{(u,v]}(A)\right)<\infty$ for every $(u,v]\subset(t,\infty)$. Write $(t,\infty)$ as $\bigcup_{i=1}^{\infty}(u_i,v_i]$ where the intervals are disjoint. 

Then $$\infty=\tau\left((A-tI)\chi_{(t,\infty)}(A)\right)=\sum_{i=1}^{\infty}\tau\left((A-tI)\chi_{(u_i,v_i]}(A)\right)$$ and since $\tau\left((A-tI)\chi_{(u_i,v_i]}(A)\right)<\infty$ for every $i$ we can split the sequence $\{(u_i,v_i]\}_1^{\infty}$ into $N$ disjoint subsequences such that by denoting $\Delta_k$ the union of the intervals in the $k^{th}$ subsequence we obtain $\tau\left((A-tI)\chi_{\Delta_k}(A)\right)=\infty$ for every $k$. Take $E_1=\chi_{[0,t]\cup \Delta_1}(A)$, $E_k=\chi_{\Delta_k}(A)$ for $k\geq 2$, and $A_k:=AE_k$. Then $A=\bigoplus_{i=1}^{N}A_i$ and $\tau((A_i-tI)_+)=\infty$ for every $i$.  

\item[ ii$\,'$) \& iii$\,'$)]  
Choose $s>0$ such that $\|A\chi_{(s, \infty)}(A)\|_e= \|A\|_e$, 
(resp., such that $\tau((A\chi_{(s, \infty)}(A)-tI)_+)=\infty$ in the case when $\tau((A-tI)_+)=\infty$ for some $t>0$). 

By ii) (resp., iii)) applied to $A\chi_{(s, \infty)}(A)$, we can find a sequence of mutually orthogonal projections $F_j\le \chi_{(s, \infty)}(A)$ that commute with $A\chi_{(s, \infty)}(A)$ and hence with $A$ and   such that $\|A\chi_{(s, \infty)}(A)F_j\|_e= \|A\chi_{(s, \infty)}(A)\|_e$ (resp., and such that $\tau((A\chi_{(s, \infty)}(A)F_j-tI)_+)=\infty)$.  Thus $\|AF_j\|_e= \|A\|_e$ (resp., $\tau((AF_j-tI)_+)=\infty$) for all $j$. Let $s_n\downarrow 0$ with $s_1=s$ and let $E_j:=F_j\oplus \chi_{(s_{j+1}, s_j]}(A)$. Then it is immediate to verify that $A_j:=AE_j$ satisfy the required conditions.
\end{proof}

The next lemma is a natural analogue of Lemma \ref{gen ess spect} when dealing with type III factors. The proof is similar but simpler than the proof of Lemma \ref{gen ess spect} and hence we omit it.
\begin{lemma} \label{gen spect}
Let $\mathcal{M}$ be type III, $A \in \mathcal{M}^+$, and $N\geq 1$ an integer or $N=\infty$. 

\item [ i)] 
If $\lambda \in \sigma(A)$ then $A=\bigoplus_{i=1}^{N}AE_i$ for some mutually orthogonal non-zero projections $E_i\in\{A\}'$ with $\lambda \in \sigma(A_{E_i})$ for every $i$.
\item [ ii)] $A=\bigoplus_{i=1}^{N}A_i$ for some $A_i\in \mathcal M^+$ with $\|A_i\|=\|A\|$ for every $i$. 
\item [ ii$\,'$)] If $N=\infty$  the positive operators $A_i$ in ii) can be chosen to be locally invertible.
\end{lemma}

 When $\mathcal{M}$ is semifinite, $A \in \mathcal{M}^+$ and $\tau((A-I)_+)=\infty$, we know from  \cite[Theorem 6.6]{SSP} that $A$ is a sum of projections. If $\mathcal{M}$ is of type I, by further decomposing these projections it follows that $A$ is a sum of equivalent projections. We need to show that the same conclusion holds also in the case when $\mathcal{M}$ is of type II$_\infty$. We will achieve this via embedding a type I$_\infty$ factor in the II$_\infty$ factor through the following reductions.
\begin{lemma} \label{dec diff}  
Let $\mathcal{M}$ be semifinite and $A\in \mathcal{M}^+$ with $\tau(A)=\infty$, then there exist an infinite sequence of mutually orthogonal projections $E_n$ with $\tau(E_n)=1$ and a sequence of positive numbers $t_n$ with $\sum_{n=1}^{\infty}t_n=\infty$, such that $A\geq \sum_{n=1}^{\infty}t_nE_n$.
\end{lemma}
\begin{proof}
If $\|A\|_e>0$, then there is some $t>0$ for which $\tau\left(\chi_{(t,\infty)}(A)\right)=\infty$ and $A\geq t\chi_{(t,\infty)}(A)$.
The conclusion is now obvious by choosing mutually orthogonal projections $E_n$ such that $\tau(E_n)=1$ and $\chi_{(t,\infty)}(A)= \sum_{n=1}^{\infty}E_n$. 

Assume now that $\|A\|_e=0$. The case when $\mathcal{M}$ is a type I$_{\infty}$ factor is immediate since $A$ is then diagonalizable. 
So we can assume that $\mathcal{M}$ is type II$_{\infty}$. For $t\geq 0$ let  $\mu_t(A)$ be the $t^{th}$ singular value of $A$, that is $$\mu_t(A)=\inf \{\,s\geq0\mid \tau\left(\chi_{(s,\infty)}(A)\right)\leq t\,\}.$$ For equivalent definitions and properties of the singular values see \cite{F}.

Let $t_1=\mu_1(A)$. Then $\|A\|\geq t_1>0$ because $\tau(A)=\infty$ implies that $\tau\left(\chi_{(0,\infty)}(A)\right)=\infty$ and $$\tau\left(\chi_{(t_1,\infty)}(A)\right)\leq 1\leq \tau\left(\chi_{[t_1,\infty)}(A)\right)<\infty,$$
the latter inequality holding because $\|A\|_e=0$. 

Therefore $$0\leq1-\tau\left(\chi_{(t_1,\infty)}(A)\right)\leq\tau\left(\chi_{\{t_1\}}(A)\right)< \infty$$
and we can find a projection $P_1\leq \chi_{\{t_1\}}(A)$ such that $\tau(P_1)=1-\tau\left(\chi_{(t_1,\infty)}(A)\right)$. 

Let  $E_1=\chi_{(t_1,\infty)}(A)+P_1$. Then $E_1$ has trace 1 and commutes with $A$. 

Apply the above step to $A_1=A(I-E_1)$ to find  $0<\mu_1(A_1)=t_2\leq t_1$ and a trace 1 projection $E_2$ such that $E_2=\chi_{(t_2,\infty)}(A_1)+P_2$ for some projection $P_2\leq \chi_{\{t_2\}}(A_1)$. Since $I-E_1=\chi_{[0,t_1)}(A)+\left(\chi_{\{t_1\}}(A)-P_1\right)$ we see that $$E_2=\chi_{(t_2,\infty)}(A)(I-E_1)+P_2=\chi_{(t_2,t_1)}(A)+P_2+\left(\chi_{\{t_1\}}(A)-P_1\right).$$ 

Reasoning in the same way we construct by induction a non-increasing sequence of positive numbers $\{t_j\}_{j=1}^{\infty}$ and two sequences of projections $\{E_j\}_{j=1}^{\infty}$ and $\{P_j\}_{j=1}^{\infty}$ such that $\tau(E_j)=1$, $P_j \leq \chi_{\{t_j\}}(A)$ and $$E_j=\chi_{(t_j,t_{j-1})}(A)+P_j+\left(\chi_{\{t_{j-1}\}}(A)-P_{j-1}\right)$$ for every $j$, where $t_0=\|A\|$ and $P_0=0$.

We have that the projections $E_j$ are mutually orthogonal, and $$\chi_{(t_n,\infty)}(A)\leq E_1+\ldots+E_n\leq \chi_{[t_n,\infty)}(A)\quad\text{for every $n$.}$$ Together with $\tau(E_1+\ldots+E_n)=n$ and $\|A\|_e=0$ this implies that $\lim_{n\to {\infty}}t_n=0$ and therefore $\sum_{j=1}^{\infty}E_j=\chi_{(0,\infty)}(A)$.

From the construction we have that $E_j\in \{A\}'$ and $$\chi_{(t_j,t_{j-1})}(A)\le E_j\le \chi_{[t_j,t_{j-1}]}(A).$$ Hence $t_jE_j\leq AE_j\leq t_{j-1}E_j$ and hence $$\sum_{j=1}^{\infty}t_jE_j\le A\le \sum_{j=1}^{\infty}t_{j-1}E_j.$$ The conclusion is now obvious.
\end{proof} 

\begin{lemma} \label{trace geq 1}
Let $\mathcal{M}$ be type II$_\infty$ and $A \in \mathcal{M}^+$ with $\tau((A-I)_+)=\infty$. Then $A=\sum_{i=1}^{\infty}R_i$ for some projections $R_i$ with $\tau(R_i)\geq1$.
\end{lemma}
\begin{proof}
We first show that $A\geq \sum_{n=1}^{\infty}P_n$ for some sequence $\{P_n\}$ of trace 1 projections.
Apply Lemma \ref{dec diff} to $(A-I)_+$ to find trace 1 mutually orthogonal projections $E_n$ and  $t_n\ge 0$ such that  $(A-I)_+\ge \sum_{n=1}^{\infty}t_nE_n$ and $\sum_{n=1}^{\infty}t_n=\infty$. Set $A_0:=(A-I)_+- \sum_{n=1}^{\infty}t_nE_n$.

Then 
\begin{align}
A&=A\chi_{(0,1]}(A)+\chi_{(1,\infty)}(A)+(A-I)_+\notag \\
&=A\chi_{(0,1]}(A)+\chi_{(1,\infty)}(A)+A_0+\sum_{n=1}^{\infty}t_nE_n\notag \\
&=A\chi_{(0,1]}(A)+A_0+\big(\chi_{(1,\infty)}(A)-\sum_{n=1}^{\infty}E_n\big)+\sum_{n=1}^{\infty}(t_n+1)E_n\notag\\
&=A_1+A_2,\notag
\end{align}
where 
\ba A_1&=A\chi_{(0,1]}(A)+A_0+\big(\chi_{(1,\infty)}(A)-\sum_{n=1}^{\infty}E_n\big)\geq 0\\A_2&=\sum_{n=1}^{\infty}(t_n+1)E_n.
\end{align*}  Then $ \tau((A_2-I)_+)=\sum_{n=1}^{\infty}t_n=\infty.$

By considering a separable, infinite dimensional Hilbert space $\mathcal{H}$ and an injective, normal, *-homomorphism of $B(\mathcal{H})$ into $\mathcal{M}$ such that the projections $E_{n}$ correspond to rank one projections in $B(\mathcal{H})$ that sum to $I$, we can apply \cite [Theorem 1.1 i)]{SSP} to conclude that $A_2=\sum_{i=1}^{\infty}P_i$ where $\tau(P_i)=1$ for all $i$. 
Hence $A\geq \sum_{i=1}^{\infty}P_i$ and $\tau(P_i)=1$ for all $i$.

Now by applying Lemma \ref{gen ess spect} $iii)$ we can write $A$ as $A=\bigoplus_{i=1}^4A_i$ such that $\tau((A_i-I)_+)=\infty$ for $1\le i\le4$. From the first part of the proof  decompose $A_i$ for $1\le i\le 2$ as $A_i=\sum_{j=1}^{\infty}P_j^i+B_i$ with projections $\tau(P_j^i)=1$ for every $j\geq 1$ and with remainder $B_i\geq 0$. Then $$A=\Big(\sum_{j=1}^{\infty}P_j^1+(B_1\oplus A_3)\Big)\oplus \Big(\sum_{j=1}^{\infty}P_j^2+(B_2\oplus A_4)\Big).$$

Since $$\tau(((B_1\oplus A_3)-I)_+)=\tau(((B_2\oplus A_4)-I)_+)=\infty,$$  
we can apply \cite[Theorem 6.6]{SSP} to conclude that $$B_1\oplus A_3=\sum_{j=1}^{\infty}Q_j^1\quad\text{and}\quad B_2\oplus A_4=\sum_{j=1}^{\infty}Q_j^2$$ for some projections $Q_j^1$ and $Q_j^2$.
Then 
\begin{align*}
A&=\Big(\sum_{j=1}^{\infty}P_j^1+\sum_{j=1}^{\infty}Q_j^1\Big)\oplus \Big(\sum_{j=1}^{\infty}P_j^2+\sum_{j=1}^{\infty}Q_j^2\Big)\\ &=\sum_{j=1}^{\infty}\left(P_j^1\oplus Q_j^2\right)+\sum_{j=1}^{\infty}\left(P_j^2\oplus Q_j^1\right).
\end{align*}
Since $\tau\left(P_j^1\oplus Q_j^2\right)\geq 1$ and $\tau\left(P_j^2\oplus Q_j^1\right)\geq 1$ for every $j$, this  concludes the proof. 
\end{proof}

\begin{proposition}\label{proj}
Let $\mathcal{M}$ be semifinite,  $0\ne P\in \mathcal{M}$ be a projection, and $A\in \mathcal{M}^+$ with 
$\tau((A-I)_+)=\infty$. Then $A=\sum_{j=1}^{\infty}P_j$ for some projections $P_j\sim P$.
\end{proposition}
\begin{proof}

Assume first that $P$ is infinite.
By applying Lemma \ref{gen ess spect} $iii)$ we can write $A=\bigoplus_{i=1}^{\infty}A_i$ with $\tau((A_i-I)_+)=\infty$ for every $i$. Using Lemma \ref{trace geq 1} when $\mathcal{M}$ is type II$_{\infty}$ and \cite[Theorem 1.1 i)]{SSP} when $\mathcal{M}$ is type I$_{\infty}$ we obtain that $A_i=\sum_{j=1}^{\infty}P_{ij}$ with $\tau(P_{ij})\geq 1$ for every $j$ and every $i$.
So $$A=\bigoplus_{i=1}^{\infty}A_i=\bigoplus_{i=1}^{\infty}
\sum_{j=1}^{\infty}P_{ij}=\sum_{j=1}^{\infty}\bigoplus_{i=1}^{\infty}P_{ij}
=\sum_{j=1}^{\infty}P_{j}$$ where $P_j=\bigoplus_{i=1}^{\infty}P_{ij}$ is an infinite projection for every $j$. Since $\mathcal{M}$ is $\sigma$-finite we conclude that $P_j\sim P$ for every $j$.

If $P$ is finite decompose first $A$ as $A=\sum_{i=1}^{\infty}Q_i$ with $Q_i$ infinite and then write each $Q_i$ as $Q_i=\sum_{j=1}^{\infty}P_{ij}$ with $P_{ij}\sim P$ for every $j$ and $i$.

\end{proof}

\begin{corollary}\label{sequence of projections}
Let $\mathcal{M}$ be semifinite,  $A\in \mathcal{M}^+$ with $\tau((A-I)_+)=\infty$, $\{P_j\}_{j=1}^{\infty}\subseteq \mathcal{M}$ be a sequence of projections where either infinitely many of the projections $P_j$ are infinite or finitely many are infinite and $\sum\{ \tau(P_j)\mid \tau(P_j)<\infty\}=\infty$. Then $A=\sum_{j=1}^{\infty}Q_j$ for some $Q_j\sim P_j$.
\end{corollary}
\begin{proof}
By passing if necessary to equivalent projections we can assume that the projections $P_j$ are mutually orthogonal. The assumption on the sequence $\{P_j\}$ guarantees that we can write $\mathbb{N}=\bigcup_{k=1}^\infty J_k$ for some infinite, disjoint sets $J_k$ such that $\sum_{j\in J_k}\tau(P_j)=\infty$ for every $k$.

By Proposition \ref {proj}, $A=\sum_{k=1}^{\infty}R_k$ with $R_k$ infinite projections for every $k$. 
But then  $R_k\sim \bigoplus_{j\in J_k}P_j$ and therefore $R_k=\sum_{j\in J_k}Q_j$ with $Q_j\sim P_j$ for $j\in J_k$, $k\ge1$.
The conclusion is now obvious. 
\end{proof}

\begin{remark}
Assume that $A=\sum_{j=1}^{\infty}Q_j$ with $Q_j\sim P_j$ where the projections   $P_j$ are infinite for $1\le j\le n$ and $\sum_{j=n+1}^{\infty}\tau(P_j)<\infty$. Then $\sum_{j=n+1}^{\infty}Q_j$ is trace class and consequently $\|A\|_e=\|Q_1+\dots+Q_n\|_e\le n$. Thus the conclusion of Corollary \ref {sequence of projections} fails if $\|A\|_e>n$.
\end{remark}

\section{Sums of equivalent copies of a sequence of operators}
It is clear that if we want to decompose a ``large" $A\in  \mathcal{M}^+$ into $A=\sum_{j=1}^\infty C_j$ for some $C_j\sim B_j$, the sequence $B_j$ cannot be ``too small". 

An obvious obstruction is that  if every $B_j$ belongs to the ideal of compact operators $\mathcal J $, $\sum_{j=1}^\infty \|B_j\|< \infty$, and $A=\sum_{j=1}^\infty C_j$ for some $C_j\sim B_j$, then $A\in \mathcal J$. The same conclusion holds if $B_j=B_j'+ B''_j$ with $B_j', B''_j\in \mathcal J^+$, 
 $\sum_{j=1}^\infty \|B'_j\|< \infty$, and $\sum_{j=1}^\infty \tau(B''_j)< \infty$.

 A natural condition on the sequence $\{B_j\}_{j=1}^{\infty}$ to avoid this obstruction is to ask that $\inf_{j\ge1}\tau(\chi_{(\beta, \infty)}(B_j))>0$ for some $\beta>0$. This is equivalent to the existence of a non-zero projection $P\prec \chi_{(\beta, \infty)}(B_j)$, that is, to the condition $B_j\ge \beta P_j$ with $P_j\sim P$.  
 Compare  with \cite [Lemma 2.4]{SMV} where the condition on $B_j$ is that $B_j\ge \beta R_{B_j}$ for some $\beta >0$ and all $j$.

 We start by showing that if $A\in  \mathcal{M}^+$ is ``large" with respect to the norms of the sequence $\{B_j\}$, then we can absorb a sequence of ``copies" of these operators into $A$ and be left with a ``large" remainder.

\begin{lemma}\label{subsume a sequence} 
Let $A\in \mathcal{M}^+$ and $\{B_j\}_{j=1}^\infty \subseteq \mathcal{M}^+$ with $\alpha:=\sup_{j\ge1}\|B_j\|<\infty$ and $\tau((A-\alpha I)_+)=\infty$ if $\mathcal{M}$ is semifinite  (resp., $\|A\|>\alpha$ when $\mathcal{M}$ is type III). Then $A=\sum_{j=1}^\infty C_j+A'$ for some $C_j\sim B_j$, with the projection $N_{C_j}$ on the null space of $C_j$ being infinite for every $j$, and $A'\ge 0$ with $\tau((A'-\alpha I)_+)=\infty$ and $\|A'\|_e=\|A\|_e$ if $\mathcal{M}$ is semifinite and $\|A'\|=\|A\|$ when $\mathcal{M}$ is type III.
\end{lemma}

\begin{proof}
If $\alpha=0$ there is nothing to prove. So we can assume that $\alpha>0$.
From Lemma \ref{gen ess spect} ii) and iii) (resp., Lemma \ref{gen spect} ii) when $\mathcal{M}$ is type III) we can decompose $A$ as $A=\bigoplus_{j=0}^\infty A_j$ with $\tau((A_j-\alpha I)_+)=\infty$ and $\|A_j\|_e=\|A\|_e$ (resp., $\|A_j\|=\|A\|$ when $\mathcal{M}$ is type III) for all $j$. There are $s_j\ge \alpha$ such that $\chi_{(s_j,\infty)}(A_j)$ is infinite.   
Since $R_{B_j}\prec I\sim \chi_{(s_j,\infty)}(A_j)$ for all $j\ge 1$, we can find a partial isometry $V_j$ with $V_j^*V_j=R_{B_j}$ and $V_jV_j^*\le \chi_{(s_j,\infty)}(A_j)$. Taking $C_j=V_jB_jV_j^*$ we see that $C_j\sim B_j$, $N_{C_j}$ is infinite and that $$C_j\le \|B_j\|V_jV_j^*\le \alpha\chi_{(s_j,\infty)}(A_j)\le s_j\chi_{(s_j,\infty)}(A_j)\le A_j.$$
Hence $A=(A_0\oplus (\sum_{j=1}^\infty A_j-C_j))+\sum_{j=1}^\infty C_j$ and by taking $$A'=A_0\oplus \Big(\sum_{j=1}^\infty (A_j-C_j)\Big)\ge 0$$ we have immediately the desired conclusion.
\end{proof}

The reason that we go to the additional step of choosing the operators $C_j$ with infinite null space in this lemma as well as in  other results of this section, is that this will enable us to pass from the equivalence relation $\sim$ to the  unitary equivalence relation $\cong$ in the proof of Lemma \ref {monotone sequence} below as well as in Section 4.

 By using decomposition of positive elements into sums of projections, we can now eliminate the remainder in Lemma   \ref {subsume a sequence} in the following special case.

\begin{lemma}\label{common proj direct summand}
Let $A\in \mathcal{M}^+$ and $\{B_j\}_{j=1}^\infty \subseteq \mathcal{M}^+$ with $\alpha:=\sup_{j\ge1}\|B_j\|<\infty$ and $\tau((A-\alpha I)_+)=\infty$ when $\mathcal{M}$ is semifinite  (resp., $\|A\|>\alpha$ when $\mathcal{M}$ is type III). Assume furthermore that there is a number  $t>0$ and a non-zero projection $P$ such that $B_j=tP_j\oplus B'_j$ for some $B'_j\ge 0$ and for projections $P_j\sim P$. Then $A=\sum_{j=1}^{\infty}C_j$ for some $C_j\sim B_j$ with $N_{C_j}$ infinite. 
\end{lemma}
\begin{proof}
Assume first that $\mathcal{M}$ is semifinite.
 By applying Lemma \ref{gen ess spect} iii) decompose $A$ as $A=\bigoplus_{i=1}^4A_i$ such that $\tau((A_i-\alpha I)_+)=\infty$ for $1\le i\le4$. 
  By Lemma \ref {subsume a sequence} we have that $A_i=\sum_{j=1}^\infty B''_{4j+1-i}+A_i'$ where $B''_{4j+1-i}\sim B'_{4j+1-i}$ for all $j$ and $\tau((A'_i-\alpha I)_+)=\infty$ for $1\le i\le4$.
Now $\tau\left((A_i'-\alpha I)_+\right)=\infty$ and $\alpha\ge t$ imply that $\tau\left(\left(\frac{1}{t}A_i'-I\right)_+\right)=\infty$ and hence
from Proposition \ref{proj}, we have that for each $i$, $\frac{1}{t}A_i'=\sum_{j=1}^{\infty}P'_{4j+i-4}$ with $P'_{4j+i-4}\sim P$.
Then,
\begin{align*}
A&=\Big(\sum_{ j=1}^\infty B''_{4j}+ t\sum_{j=1}^\infty P'_{4j-3}\Big)\oplus\Big(\sum_{ j=1}^\infty B''_{4j-1}+ t\sum_{ j=1}^\infty P'_{4j-2}\Big)\\&\oplus \Big(\sum_{j=1}^\infty B''_{4j-2}+ t\sum_{j=1}^\infty P'_{4j-1}\Big)\oplus \Big(\sum_{j=1}^\infty B''_{4j-3}+ t\sum_{j=1}^\infty P'_{4j}\Big)\\
&=\Big(\Big(\sum_{j=1}^\infty B''_{4j}\oplus \sum_{j=1}^\infty tP'_{4j}\Big)+\Big(\sum_{j=1}^\infty B''_{4j-1}\oplus  \sum_{j=1}^\infty tP'_{4j-1}\Big)\Big)\\
&\oplus\Big(\Big(\sum_{j=1}^\infty B''_{4j-2}\oplus \sum_{j=1}^\infty tP'_{4j-2}\Big)+\Big(\sum_{j=1}^\infty B''_{4j-3}\oplus  \sum_{j=1}^\infty tP'_{4j-3}\Big)\Big)\\
&=\Big(\sum_{j=1}^\infty C_{4j}+ \sum_{j=1}^\infty C_{4j-1}\Big)\oplus\Big(\sum_{j=1}^\infty C_{4j-2}+ \sum_{j=1}^\infty C_{4j-3}\Big)
\end{align*}
where $C_j=B''_j\oplus tP'_j\sim B_j$ for all $j$. Finally, notice that both $$\Big(\sum_{j=1}^\infty C_{4j}+ \sum_{j=1}^\infty C_{4j-1}\Big) \text{ and } \Big(\sum_{j=1}^\infty C_{4j-2}+ \sum_{j=1}^\infty C_{4j-3}\Big)$$ have infinite trace and hence infinite rank, which proves that $N_{C_j}$ is infinite for all $j$.

When $\mathcal{M}$ is type III the proof is similar, the only difference being that we must replace the essential norm with the operator norm, Lemma \ref{gen ess spect} with Lemma \ref{gen spect} and Proposition \ref{proj} with \cite[Theorem 1.1 (iii)]{SSP}. 
 \end{proof}

The condition in Lemma \ref {common proj direct summand} that  each $B_j$ has a direct summand  $tP_j$ with $P_j\sim P\ne 0$ is of course too limiting. A more natural condition and closer  to the spirit of the one in Lemma \ref {Lemma 2.4} (i.e.,   \cite [Lemma 2.4]{SMV}) is that $B_j\ge tP_j$ with $P_j\sim P\ne 0$. In Theorem \ref{equiv infinite trace} we are going to prove that this condition is indeed sufficient and in fact, that it is enough to require it for infinitely many indices. The core of the argument is the following approximation lemma.
\begin{lemma}\label{monotone sequence}
Let $A\in \mathcal{M}^+$ and $\{B_j\}_{j=1}^\infty \subseteq \mathcal{M}^+$ with $\alpha:=\sup_{j\ge1}\|B_j\|<\infty$ and $\tau((A-\alpha I)_+)=\infty$ when $\mathcal{M}$ is semifinite  (resp., $\|A\|>\alpha$ when $\mathcal{M}$ is type III).  Assume that there is $\beta>0$ such that $\chi_{(\beta,\infty)}(B_j)$ is infinite for every $j$. Then for every $\epsilon >0$,  $A=\sum_{j=1}^{\infty}C_j+R$ for some $C_j\sim B_j$ with $N_{C_j}$ infinite and $0\le R\le \epsilon A$.
 \end{lemma}

\begin{proof}
By choosing a sequence of partial isometries $V_j$  with $V_j^*V_j=R_{B_j}$ and mutually orthogonal  range projections $V_jV_j^*$, we can replace  $B_j$ with the equivalent elements $B'_j:=V_jB_jV_j^*$ which have  mutually orthogonal range  projections $R_{B'_j}$. Since $R_{B'_j}\ge \chi_{(\beta,\infty)}(B'_j)\sim\chi_{(\beta,\infty)}(B_j)$ is infinite, we conclude that $N_{B'_j}$ is infinite for every $j$. Thus to simplify notations, assume directly that $N_{B_j}$ is infinite for every $j$.

First we show that there is a subsequence $\{j_k\}_{k=1}^{\infty}$ and a monotone sequence $\{\gamma_{j_k}\}_{k=1}^{\infty}\subseteq (\beta,\alpha]$ with  $t:=\lim_{k\to \infty} \gamma_{j_k}$ and $\chi_{(\gamma_{j_k},t]}(B_{j_k})$ (resp., $\chi_{(t,\gamma_{j_k}]}(B_{j_k})$) is infinite for all $k$ if $\{\gamma_{j_k}\}$ is monotone increasing (resp., decreasing).

To see this, choose for every $j$ an interval $(\beta_j, \alpha_j]\subset (\beta, \alpha]$ of length $\le \frac{1}{2^{j}}$ and such that $ \chi_{(\beta_j, \alpha_j]}(B_j)$ is infinite. Then choose a subsequence $\{j_k\}$ such that both $\{\beta_{j_k}\}$  and $\{\alpha_{j_k}\}$ are monotone and set $ t:=\lim_k \beta_{j_k}=\lim_k\alpha_{j_k}$. 
 If both $\{\beta_{j_k}\}$  and $\{\alpha_{j_k}\}$ are increasing, then set $\gamma_{j_k}:=\beta_{j_k}$, if both are decreasing set $\gamma_{j_k}:=\alpha_{j_k}$. If $\{\beta_{j_k}\}$  is increasing and $\{\alpha_{j_k}\}$
 is decreasing, then 
 $$\chi_{(\beta_{j_k},\alpha_{j_k}]}(B_{j_k})=\chi_{(\beta_{j_k},t]}(B_{j_k})+\chi_{(t, \alpha_{j_k}]}(B_{j_k})$$
hence for each $k$ at least one of the projections $\chi_{(\beta_{j_k},t]}(B_{j_k})$ and $\chi_{(t, \alpha_{j_k}]}(B_{j_k})$ must be infinite. Thus by passing if necessary to a subsequence we can assume that either the first projection is always infinite or the second projection is always infinite and define $\gamma_{j_k}$ accordingly.

To simplify notations, by invoking  Lemma \ref{subsume a sequence} we can assume
 that $j_k=k$ and that $0< \frac{|\gamma_1-t |}{\min\{t,\gamma_1\} } <\epsilon$. 

Define 
$
\begin{cases}
t_j=\gamma_j, s_j=t, s=t & \text{when $\{\gamma_j\}$ is increasing}\\
t_j=t, s_j=\gamma_j, s=\gamma_{1} & \text{when $\{\gamma_j\}$ is decreasing}
\end{cases}.
$\\
Then $0<\beta\le t_j<s_j\le s\le \alpha$,  $\chi_{(t_{j},s_{j}]}(B_{j})$ is infinite for every $j$, and
$$
\sup_{j\ge1}\frac{s-t_{j}}{s}=
\begin{cases}
\frac{t-\gamma_{1}}{t} & \text{when $\{\gamma_j\}$ is increasing}\\
\frac{\gamma_1-t}{\gamma_1} &  \text{when $\{\gamma_j\}$ is decreasing}\end{cases}
\le \frac{|\gamma_{1}-t|}{\min\{t,\gamma_{1}\}}<\epsilon.
$$
Now set $P_{j}:=\chi_{(t_{j},s_j]}(B_{j})$ and $D_{j}:=sP_{j}\oplus B_{j}(I-P_{j}).$
Then 
$$0\le D_{j}-B_{j}=(sI-B_{j})P_{j}\le(s-t_{j})P_{j}\quad\forall\, j.$$

By construction, $R_{D_{j}}=R_{B_{j}}$ and hence $N_{D_{j}}= N_{B_{j}}$ is infinite for every $j$. 

The hypotheses of Lemma \ref{common proj direct summand} are satisfied for $A$ and the sequence $\{D_{j}\}_{j=1}^{\infty}$, hence $A=\sum_{j=1}^{\infty} C'_{j}$ with $C'_{j}\sim D_{j}$ and $N_{C'_{j}}$ infinite for every $j$.
But then $C'_{j}\cong D_{j}$ for every $j$, i.e.
$C'_{j}=U_{j}D_{j}U_{j}^*$ for some unitaries $U_{j}$. 
Taking $C_{j}:=U_{j}B_{j}U_{j}^*$, $N_{C_{j}}$ is infinite, and $C_{j} \sim B_j$.

Since $C'_{j}-C_{j}=U_{j}(D_{j}-B_{j})U_{j_i}^*$ we have
$$
0\le C'_{j}-C_{j}\le 
U_{j}(s-t_j)P_{j}U_{j}^*
\le \sup_{j\ge1}\frac{s-t_{j}}{s}U_{j}sP_{j}U_{j}^*\le \epsilon U_jD_jU_j^* =\epsilon C'_j.
$$

Therefore $R:=\sum_{j=1}^{\infty}(C'_j-C_j)$
is strong operator convergent and hence so is $\sum_{j=1}^{\infty}C_{j}$. Then  $A=\sum_{j=1}^{\infty}C_{j}+R$ and$$0\le R =\sum_{j=1}^{\infty}(C'_j-C_j)\le \epsilon \sum_{j=1}^{\infty}C'_j=\epsilon A.$$
\end{proof}

 \begin{theorem}\label{equiv infinite trace}
Let $\mathcal{M}$ be a $\sigma$-finite infinite factor, $A\in \mathcal{M}^+$ and $\{B_j\}_{j=1}^\infty \subseteq \mathcal{M}^+$ with $\alpha:=\sup_{j\ge1}\|B_j\|<\infty$. 
Each of the following two conditions is sufficient  for $A=\sum_{j=1}^{\infty}C_j$ for some $C_j\sim B_j$ with $N_{C_j}$ infinite:
\item [ i)] $\tau((A-\alpha I)_+)=\infty$ when $\mathcal{M}$ is semifinite (resp., $\|A\|>\alpha$ when $\mathcal{M}$ is type III) and 
there are a $\beta>0$ and a non-zero projection $P$ for which $P\prec \chi_{(\beta,\infty)}(B_{j})$ for infinitely many indices $j$. 
\item [ ii)]
 $\|A\|_e\ge\alpha$ when $\mathcal{M}$ is semifinite (resp., $\|A\|\ge\alpha$ when $\mathcal{M}$ is type III), 
$\chi_{\{\alpha\}}(B_j)=0$ for all $j$ and 
 there are $0<\beta<\gamma<\alpha$ and a non-zero projection $P$ such that $ P\prec \chi_{(\beta,\gamma]}(B_j)$ for infinitely many indices $j$.
 \end{theorem}

\begin{proof}
\item [ i)] Reasoning as in the proof of Lemma \ref {monotone sequence}, assume without loss of generality that the elements $B_j$ have mutually orthogonal projections and by invoking Lemma \ref {subsume a sequence} that the condition  $P \prec \chi_{(\beta,\infty)}(B_{j})$ holds for all $j$.  

Next, partition  $\mathbb{N}=\bigcup_{i=1}^{\infty}N_i$ into a collection of disjoint infinite sets $N_i$  and set $B'_i=\bigoplus_{j\in N_i}B_j$. Then $\chi_{(\beta,\infty)}(B'_i)$ is infinite for every $i$. If we find a decomposition $A= \sum_{i=1}^\infty C'_i$ for some $C'_i\sim B'_i$ and $N_{C'_i}$ infinite for all $i$, we can then  refine it into a decomposition $A=\sum_{j=1}^\infty C_j$ for some $C_j\sim B_j$ with $N_{C_j}$ infinite for every $j$. Thus to simplify notations assume directly that $\chi_{(\beta,\infty)}(B_{j})$ is infinite for every $j$.

Using Lemma \ref{gen ess spect} iii) (resp., Lemma \ref{gen spect} ii) when $\mathcal{M}$ is type III) decompose $A$ as $A=\bigoplus_{k=1}^{\infty}A_k$ with $\tau((A_k-\alpha I)_+)=\infty$ (resp., $\|A_k\|>\alpha$ when $\mathcal{M}$ is type III). Decompose $\mathbb{N}$ as $\mathbb{N}=\bigcup_{k=1}^{\infty}J_k$ for some disjoint, infinite subsets $J_k$. 

Apply Lemma \ref{monotone sequence} to $A_1$ and $\{B_j\}_{j\in J_1}$ with $\epsilon=1$ 
to obtain that $$A_1=\sum_{j\in J_1}C_j+R_1$$ for some $C_j\sim B_j$  for $j\in J_1$ with infinite $N_{C_j}$  and $0\le R_1\le A_1$.
Then $R_1\perp A_2$ and $\tau((R_1+A_2-\alpha I)_+)=\infty$ if $M$ is semifinite (or $\|R_1+A_2\| > \alpha$ if $M$ is type III). Thus we can apply again Lemma \ref{monotone sequence} to $R_1+A_2$ and $\{B_j\}_{j\in J_2}$ with $\epsilon=\frac{1}{2}$ and obtain that
$R_1+A_2= \sum_{j\in J_2}C_j+R_2$ for some $C_j\sim B_j$  for $j\in J_2$ with infinite $N_{C_j}$  and $0\le R_2\le \frac{1}{2}A_2$. Then  $$A_1+A_2= \sum_{k=1}^2\sum_{j\in J_2}C_j+R_2.$$

Iterating, we find $\{C_{j}\}_{j\in J_k}$ with $C_{j}\sim B_{j}$ and $N_{C_j}$ infinite, $0\le R_k\le \frac{1}{k}A_k\le \frac{1}{k}A$ such that for every $n$, $$\sum_{k=1}^nA_k= \sum_{k=1}^n \sum_{j\in J_{k+1}}C_{j}+R_{n+1}.$$
Since $\|R_n\|\to 0$ this concludes the proof of part i).

\item [ ii)]
Decompose $(\gamma,\alpha]$ as $(\gamma,\alpha]=\bigcup_{k=1}^{\infty}(t_k,t_{k+1}]$ for some sequence $\{t_k\}_{k=1}^{\infty}$ strictly increasing to $\alpha$.

Let $J:=\{j\mid P\prec \chi_{(\beta,\gamma]}(B_j)\}$ and decompose $\mathbb{N}$ as $\mathbb{N}=\bigcup_{k=1}^{\infty}J_k$ for some disjoint, infinite subsets $J_k$ with $J_k\cap J$ infinite for every $k$.

For each $k$, define $$B^j_k=\begin{cases}
B_j\chi_{(t_k,t_{k+1}]}(B_j)\oplus B_j\chi_{[0,\gamma]}(B_j) & \text{ if $j\in J_k$},\\
B_j\chi_{(t_k,t_{k+1}]}(B_j) &\text{ if $j\notin J_k$}.
\end{cases}
$$
Then for every $k\ge1$, $\sup_{j\ge1}\|B^j_k\|\le t_{k+1}<\alpha$ and $P\prec \chi_{(\beta,\gamma]}(B^j_k)$ for all $j\in J\cap J_k$.
Since $\chi_{\{\alpha\}}(B_j)=0$ for every $j$ we obtain that $$B_j=B_j\chi_{[0,\alpha)}(B_j)=\bigoplus_{k=1}^{\infty}B^j_k.$$

Using Lemma \ref{gen ess spect} ii) (resp.,  Lemma \ref{gen spect} ii) when $\mathcal{M}$ is type III,) decompose $A$ as $A=\bigoplus_{k=1}^{\infty}A_k$ with $\|A_k\|_e=\|A\|_e$ (resp.,  $\|A_k\|=\|A\|$ when $\mathcal{M}$ is type III.)
Then apply i) to $A_k$ and $\{B^j_k\}_{j\ge1}$. Hence $A_k=\sum_{j=1}^{\infty}C^j_k$ with $C^j_k\sim B^j_k$  for each $k$ and hence
$$
A=\bigoplus_{k=1}^{\infty}A_k=
\bigoplus_{k=1}^{\infty}\sum_{j=1}^{\infty}C^j_k
=\sum_{j=1}^{\infty}\bigoplus_{k=1}^{\infty}C^j_k=
\sum_{j=1}^{\infty}C_j,
$$
where $C_j=\bigoplus_{k=1}^{\infty}C^j_k\sim
\bigoplus_{k=1}^{\infty}B^j_k=B_j$.

The fact that $N_{C_j}$ can be chosen infinite follows from writing $A=A_1\oplus A_2$ with $\|A_1\|_e=\|A_2\|_e=\|A\|_e$ (resp., $\|A_1\|=\|A_2\|=\|A\|$ when $\mathcal{M}$ is type III) and splitting $\mathbb{N}$ as $\mathbb{N}=J'_1\cup J'_2$ with $J'_1\cap J'_2=\varnothing$ and $J'_1\cap J$ and $J'_2\cap J$ infinite. Then apply the first part of the proof to $A_1$ and $\{B_j\}_{j\in J'_1}$ obtaining $A_1= \sum_{j\in J'_1}C^1_j$ with $C^1_j\sim B_j$ for $j\in J'_1$. Then $C^1_j\le R_{A_1}$, hence $N_{C^1_j}\ge R_{A_2}$ is infinite. In the same way we obtain a decomposition of $A_2=\sum_{j\in J'_2}C^2_j$ with $C^2_j\sim B_j$ and $N_{C_j}$ infinite for $j\in J'_2$.
\end{proof}

Notice that if $\|A\|_e>\alpha$ then $\tau((A-\alpha I)_+)= \infty$ and therefore i) applies. But if  $\|A\|_e=\alpha$ and $\tau((A-\alpha I)_+)< \infty$ the conditions  on the sequence $\{B_j\}$ given in i) need indeed to be strengthened. For instance if $\mathcal{M}=B(\mathcal{H})$, $B_j$ are rank one projections (and hence satisfy i) but not ii)), and $A= I+K$ with $K$ a positive trace-class operator with $\tau (K)\not\in \mathbb N$, and hence $\|A\|_e=1$, then by \cite[Theorem 1.1 (i)]{SSP} $A$ cannot be a sum of projections.

\

We conclude this section by recalling the connection between decompositions of operators into sums of projections and block diagonal forms which was established in \cite[Proposition 3.1]{SSP} and \cite[Proposition 5.1] {FSP} and which can easily be extended as follows to decompositions of positive operators (see also \cite{KL} for the $B(\mathcal{H})$ rank-one projection case).

\begin{proposition}\label{block diagonal decomp}
Let $\mathcal{M}$ be a properly infinite von Neumann algebra, $A\in \mathcal{M}^+$ and $\{B_j\}_{j=1}^N\subseteq \mathcal{M}^+$, where $N\geq1$ is an integer or $N=\infty$.
The following are equivalent:

\item[ i)] $A=\sum_{j=1}^{N}C_j$ for some $C_j\sim B_j$.
\item[ ii)] There exist mutually orthogonal projections $\{E_j\}_{j=1}^N$ and a partial isometry $V$ such that $\sum_{j=1}^NE_j\ge VV^*$, $V^*V\ge R_A$ and  $E_jVAV^*E_j\sim B_j$.

If in addition $N_A\sim I$ then in ii) we can take $V=I$ .
\end{proposition}
\begin{proof}
\item [ii)] $\Rightarrow$ i) 
Let $C_j=A^{\frac{1}{2}}V^*E_jVA^{\frac{1}{2}}$. 
Then $C_j\sim E_jVAV^*E_j\sim B_j$ and $$\sum_{j=1}^NC_j=A^{\frac{1}{2}}V^*\Big(\sum_{j=1}^NE_j\Big)VA^{\frac{1}{2}}=A^{\frac{1}{2}}V^*VA^{\frac{1}{2}}=A.$$ 
\item [i)] $\Rightarrow$ ii)
Since $\mathcal{M}$ is properly infinite we can find projections $\{E_j\}_{j=0}^N$ such that $E_j\sim I$ and $\sum_{j=0}^NE_j=I$. Choose isometries $\{W_j\}_{j=1}^N$ such that $W_jW_j^*=E_j$ and define $X_j=W_jC_j^{\frac{1}{2}}$, $j\ge1$.

Just as in the proof of \cite [Proposition 3.1]{SSP} where the $C_j$ are rank one projections, one can verify that if $N$ is infinite the series $\sum_{j=1}^NX_j$ is strong operator convergent to some operator $X$ and that $X^*X=A$. Then let $X=VA^{\frac{1}{2}}$ be the polar decomposition of $X$. In particular, $V^*V=R_A$. Since $E_jX=X_j$ and $VAV^*=XX^*$, it follows that  $VV^*=R_X\le \sum_{j=1}^NE_j$ and $$E_jVAV^*E_j=E_jXX^*E_j=X_jX_j^*=W_jC_jW_j^*\sim C_j\sim B_j.$$

Finally, notice  that 
$$ I- VV^*\ge  I-\sum_{j=1}^N E_j=E_0\sim I.$$ Thus if also $I-V^*V=N_A\sim I$, $V$ can be extended to a unitary $U$ and then the conclusion follows by taking the projections $U^*E_jU$ instead of $E_j$ for $j\ge2$ and $U^*E_1\oplus E_0U$ instead of $E_1$.  
\end{proof}

We collect bellow a few remarks that are easy consequences of the proof.
\begin{remark}
\item[ 1)] We have actually proved  that i) is equivalent to asking  that $\sum_{j=1}^NE_j=I$, $E_j\sim I$, $V^*V=R_A$ and $E_jVAV^*E_j\sim B_j$.

We could equally well have found projections $\{E_j\}_{j=1}^N$ and a partial isometry $V$ such that $\sum_{j=1}^NE_j\ge VV^*$, $V^*V=R_A$, $R_{E_jVAV^*E_j}=E_j$ and $E_jVAV^*E_j\sim B_j$. 

\item[ 2)] If $\mathcal{M}$ is an infinite factor and i) holds then for any sequence of mutually orthogonal projections $\{E_j\}_{j=1}^N$ such that $R_{B_j}\prec E_j$ we can find a partial isometry $V$ such that $\sum_{j=1}^NE_j\ge VV^*$, $V^*V=R_A$ and $E_jVAV^*E_j\sim B_j$.
\end{remark}
We further notice that in the case when $\mathcal M= B(\mathcal{H})$ and all the projections $E_j$ and operators $C_j$ have rank-one, then the partial isometry $V$ plays an important role in frame theory.  Indeed the vectors $x_j\in \mathcal{H}$ for which $C_j=x_j\otimes x_j$ form a frame when $A$ is invertible (a Bessel sequence when it is not), and $V$ then coincides with the frame transform (also called analysis operator) of the Parseval frame associated with $\{x_j\}$. 

The following example shows that we cannot expect  to be able to choose $V$ unitary without some further hypotheses on $A$.
\begin{example}
Let $A= 2I$ and $B_j$ be a sequence of equivalent nonzero projections. Then  $A=\sum_{j=1}^\infty C_j$ with $C_j\sim B_j$ but of course $E_jAE_j= 2E_j \not\sim B_j$. The same conclusion holds for every $A=\lambda I$ and $\lambda>1$ by \cite [Theorem 1.1]{SSP}.
 \end{example}
 Thus  combining Theorem \ref {equiv infinite trace} and Proposition \ref {block diagonal decomp} we obtain the following form of the ``pinching theorem" of \cite {PT} for the case of positive operators in von Neumann factors.

\begin{corollary}\label{pinching corollary}
Let $\mathcal{M}$ be a $\sigma$-finite infinite factor, $A\in \mathcal{M}^+$ and $\{B_j\}_{j=1}^\infty \subseteq \mathcal{M}^+$ with $\alpha:=\sup_{j\ge1}\|B_j\|<\infty$. 
Assume one of the following two conditions holds:
\item [ i)] $\tau((A-\alpha I)_+)=\infty$ when $\mathcal{M}$ is semifinite (resp., $\|A\|>\alpha$ when $\mathcal{M}$ is type III) and 
there are a $\beta>0$ and a non-zero projection $P$ for which $P\prec \chi_{(\beta,\infty)}(B_{j})$ for infinitely many indices $j$. 
\item [ ii)]
 $\|A\|_e\ge\alpha$ when $\mathcal{M}$ is semifinite (resp., $\|A\|\ge\alpha$ when $\mathcal{M}$ is type III), 
$\chi_{\{\alpha\}}(B_j)=0$ for all $j$ and 
 there are $0<\beta<\gamma<\alpha$ and a non-zero projection $P$ such that $ P\prec \chi_{(\beta,\gamma]}(B_j)$ for infinitely many indices $j$.

Then there exist mutually orthogonal projections $\{E_j\}_{j=1}^\infty$ with $\sum_{j=1}^\infty E_j=I$, $E_j\sim I$, and a partial isometry $V$ such that  $V^*V\ge R_A$ and  $E_jVAV^*E_j\sim B_j$.
If in addition $N_A\sim I$ then we can take $V=I$ .
\end{corollary}

\section{Sums of unitary equivalent conjugates of a sequence of operators}

To obtain decompositions into sums of unitary equivalent conjugates of a sequence  $\{B_j\}_{j=1}^\infty$, we follow the approach of Bourin and Lee in \cite{UE} for the case when $\mathcal {M}= B(\mathcal{H})$ and $B_j=B$ for all $j$. In particular, in the first step here below (see \cite[Theorem 1.1 case I.(1)]{UE}) which deals with the case when $A$ is invertible, we can use their  lemmas which apply without changes to the factor case. For completeness,  we sketch the proof.
\begin{proposition}\label{unit equiv, invertible}
Let $A\in \mathcal{M}^+$ be invertible and $\{B_j\}_{j=1}^{\infty}\subseteq \mathcal{M}^+$ 
be a sequence with $\alpha:=\sup_{j\ge 1}\|B_j\|<\infty$ and such that there is a $\beta>0$ and a non-zero projection $P$ for which $P\prec \chi_{(\beta,\infty)}(B_{j})$ for infinitely many $j$. If
 $$\begin{cases}\alpha< \|A\|_e,~ 0\in \sigma_e(B_j)&\text{when $\mathcal{M}$ is semifinite,}\\
\alpha< \|A\|, ~0\in \sigma(B_j)&\text{when $\mathcal{M}$ is type III,}\end{cases} $$
then $A=\sum_{j=1}^{\infty}C_j$ for some $C_j\cong B_j$.
\end{proposition}
\begin{proof}

\textbf{Step 1:} (see \cite[Lemma 2.4]{UE}) We first show that for every $\epsilon>0$, we can decompose $A=\sum_{j=1}^{\infty}C_j+R$ for some $C_j\cong B_j$ and $0\le R\le \epsilon I$.
Choose $0<\rho< \frac{\epsilon}{2}$ such that $A-\rho I\ge 0$ and $\|A-\rho I\|_e>\alpha$ when $\mathcal{M}$ is semifinite (resp., $\|A-\rho I\|>\alpha$ when $\mathcal{M}$ is type III).

Using Theorem \ref{equiv infinite trace}, decompose $A-\rho I$ as $A-\rho I=\sum_{j=1}^{\infty}C'_j$ with $C'_j\sim B_j$ and $N_{C'_j}$ infinite. 

Since $N_{C'_j}$ is infinite,  $0\in \sigma_e(C'_j)$ when $\mathcal{M}$ is semifinite (resp., $0\in \sigma(C'_j)$ when $\mathcal{M}$ is type III).  The result of  \cite[Lemma 2.3]{UE} holds also for factors with a similar proof, that is,  there exist  $ C_j\cong B_j$ such that $\| C'_j-C_j\| \le \frac{\rho}{2^{j}}$. Thus $\sum_{j=1}^{\infty}\|C'_j-C_j\|\le \rho.$ Let $R:= \rho I+\sum_{j=1}^{\infty}(C'_j-C_j)$. Then 
$A=\sum_{j=1}^{\infty}C_j+ R$ and $0\le R \le \epsilon I.$

\textbf{Step 2:} 
The result of \cite[Lemma 2.2]{UE} is true also in factors with a similar proof. Hence we can decompose $A$ as $A=\sum_{k=1}^{\infty}A_k$ with $A_k$ invertible and $\|A_k\|_e>\alpha$ or (resp., $\|A_k\|>\alpha$ if $\mathcal{M}$ is type III).

Write $\mathbb{N}=\bigcup_{k=1}^{\infty}J_k$ for some infinite, disjoint subsets such that the sequences $\{B_j\}_{j\in J_k}$ have the same property as $\{B_j\}_{j=1}^{\infty}$. 

Apply Step 1 to conclude that $A_1=\sum_{j\in J_1}C_j+R_1$ with $C_j\cong B_j$ for $j \in J_1$ and $0\le R_1\le I$. Apply now Step 1 to $A_2+R_1$ to conclude that $A_2+R_1=\sum_{j\in J_2}C_j$ with $C_j\cong B_j$ for $j\in J_2$ and $0\le R_2\le \frac{1}{2}I$.

Continuing in the same way we construct $\{R_k\}_{k=1}^{\infty}$ such that $0\le R_k\le \frac{1}{k}I$ and $A_{k}+R_{k+1}=\sum_{j\in J_{k+1}}C_j+R_{k+1}$ with $C_j\cong B_j$ for $j\in J_k$. Just as in the last part of the proof of Theorem \ref{equiv infinite trace} it is easy to conclude that $A=\sum_{j=1}^{\infty}C_j$ with $C_j\cong B_j$ for every $j$.
\end{proof}

By Lemma \ref {gen ess spect} and  \ref {gen spect} we know that we can decompose $A$ into a direct sum of locally invertible summands
``with the same properties". The next lemma shows that we  can also decompose  each element $B_j$ into a direct sum ``with the same properties".

\begin{lemma}\label{dec of sequence}
Let $\{B_j\}_{j=1}^{\infty}\subseteq \mathcal{M}^+$ with $\alpha:=\sup_{j\ge 1}\|B_j\|<\infty$ and assume that for every $j$, $0\in \sigma_e(B_j)$ when $\mathcal{M}$ is semifinite (resp., $0\in \sigma(B_j)$ when $\mathcal{M}$ is type III) and that  there is a $\beta>0$ and a non-zero projection $P$ such that 
 $P\prec \chi_{(\beta,\infty)}(B_{j})$ for infinitely many integers $j$.

Then for every $j$ there is a sequence of mutually orthogonal infinite projections $\{F^j_k\}_{k=1}^{\infty}\subseteq \{B_j\}'$ such that $\sum_{k=1}^{\infty}F^j_k=I$, $0\in \sigma_e\big((B_j)_{F^j_k}\big)$ when $\mathcal{M}$ is semifinite (resp., $0\in \sigma\big((B_j)_{F^j_k}\big)$ when $\mathcal{M}$ is type III), and for every $k$, $P\prec \chi_{(\beta,\infty)}(B_jF^j_k)$ for infinitely many integers $j$.

If furthermore $\chi_{\{\alpha\}}(B_j)=0$ for every $j$ and for some $0<\beta<\gamma<\alpha$,  $P\prec \chi_{(\beta,\gamma]}(B_{j})$ for infinitely many integers $j$, then the projections $\{F^j_k\}_{k=1}^{\infty}$ can be chosen so that for every $k$, $\sup_{j\ge1}\|B_jF^j_k\|<\alpha$ and 
$P\prec \chi_{(\beta,\gamma]}(B_jF^j_k)$ for infinitely many integers $j$. 
\end{lemma}
\begin{proof}

We prove the lemma when $\mathcal{M}$ is semifinite and leave to the reader the similar proof of the case when $\mathcal{M}$ is type III.

Let $J:=\{\, j\mid P\prec\chi_{(\beta, \infty)}(B_j)\,\}$. Decompose $\mathbb{N}=\bigcup_{k=1}^{\infty}J_k$ such that the sets $J_k$ are disjoint and $J_k\cap J$ is infinite for every $k$.  

For every $j\ge 1$ let $F_j=\chi_{[0,\beta]}(B_j)$. From the hypothesis that $0\in \sigma_e(B_j)$ it  follows that $F_j$ is infinite and $0\in \sigma_e\left((B_j)_{F_j}\right)$.  Applying Lemma \ref{gen ess spect} i) to $\mathcal{M}_{F_j}$ and $(B_j)_{F_j}$,  decompose each $B_jF_j$ as $$B_jF_j=\bigoplus_{k=1}^{\infty} B_j\tilde F^j_k$$ for some sequences of mutually orthogonal infinite projections $\{\tilde F_k^j\}_{k\ge1}\subseteq \{B_jF_j\}'$, with $\sum_{k=1}^{\infty}\tilde F_k^j=F_j$ and $0\in \sigma_e\left((B_j)_{\tilde F^j_k}\right)$. 

For every $k$, define 
$$F_k^j=\begin{cases}
\tilde F_k^j&\text{if $j\notin J_k$},\\
\tilde F_k^j \oplus (I-F_j)&\text{if $j\in J_k$}.
\end{cases}
$$

Then $\sum_{k=1}^{\infty}F_k^j=I$, $\{F_k^j\}_{k=1}^{\infty}\subseteq \{B_j\}'$, $F_k^j$ is infinite, and $0\in \sigma_e\left((B_j)_{F^j_k}\right)$ for every $k$ and $j$. 

Finally, notice that $ \chi_{(\beta,\infty)}(B_j)=\chi_{(\beta,\infty)}(B_jF_k^j)$ for $j\in J_k$ and therefore  $P\prec \chi_{(\beta,\infty)}(B_jF_k^j)$ for all $j\in J\cap J_k$ which concludes the first part of the proof.

Assume now that $\chi_{\{\alpha\}}(B_j)=0$ for every $j$ and $J:=\{\, j\mid P\prec\chi_{(\beta, \gamma]}(B_j)\,\}$ is infinite. Then define $J_k$, $F_j$ and $\{\tilde F_k^j\}_{k\ge1}$ just like above.
Let $\{t_k\}_{k=1}^{\infty}$ be a strictly increasing sequence to $\alpha$ with $t_1=\gamma$, and  define  for every $k$:
$$F_k^j=\begin{cases}
\tilde F_k^j\oplus\chi_{(t_k,t_{k+1}]}(B_j)&\text{if $j\notin J_k$},\\
\tilde F_k^j \oplus \chi_{(t_k,t_{k+1}]}(B_j)\oplus\chi_{(\beta,\gamma]}(B_j)&\text{if $j\in J_k$}.
\end{cases}
$$

Then  $\sum_{k=1}^{\infty}F^j_k=I$ because $\chi_{\{\alpha\}}(B_j)=0$, $\sup_{j\ge1}\|B_jF^j_k\|\le t_{k+1}<\alpha$ for every $k$, and   
$P\prec \chi_{(\beta,\gamma]}(B_jF_k^j)$ for all $j\in J\cap J_k$.
\end{proof}
\begin{theorem}\label{unit equiv sequence}
Let $\mathcal{M}$ be a $\sigma$-finite, infinite factor, $A$, $B\in \mathcal{M}^+$, and $\{B_j\}_{j=1}^{\infty}\subseteq \mathcal{M}^+$ with $\alpha:=\sup_{j\ge 1}\|B_j\|<\infty$.  
Assume that  for all $j$, $N_A\prec N_{B_j}$ and $0\in \sigma_e(B_j)$ when $\mathcal{M}$ is semifinite (resp., $0\in \sigma(B_j)$ when $\mathcal{M}$ is type III).  Each of the following two conditions is sufficient  for $A=\sum_{j=1}^{\infty}C_j$ for some $C_j\cong B_j$:
\item[ i)] $\|A\|_e>\alpha$ when $\mathcal{M}$ is semifinite (resp., $\|A\|>\alpha$ when $\mathcal{M}$ is type III) and there are a  $\beta>0$ and a non-zero projection $P$ for which  $P\prec \chi_{(\beta,\infty)}(B_{j})$ for infinitely many integers $j$.

\item[ ii)] $\|A\|_e\ge\alpha$ when $\mathcal{M}$ is semifinite (resp., $\|A\|\ge\alpha$ when $\mathcal{M}$ is type III), $\chi_{\{\alpha\}}(B_j)=0$ for every $j$, and there are $0<\beta<\gamma<\alpha$ and a non-zero projection $P$ for which $P\prec \chi_{(\beta,\gamma]}(B_{j})$ for infinitely many integers $j$.
\end{theorem}
\begin{proof}

By Theorem \ref{equiv infinite trace} it follows that $A=\sum_{j=1}^\infty C_j$ for some $C_j\sim B_j$ with $N_{C_j}$ infinite. If $N_A$ is infinite then $N_{B_j}$ is infinite for all $j$ and hence $N_{B_j}\sim N_{C_j}$  and thus $C_j\cong B_j$ for every $j$.

Assume henceforth that $N_A$ is finite. Then $N_A\sim R_j$ for some $R_j\le N_{B_j}$ and since in a factor finite equivalent projections are unitarily equivalent we obtain $N_A\cong R_j$. But then $R_A\cong R_j^{\perp}\ge R_{B_j}$ and hence $R_{B_j}\cong Q_j$ for some $Q_j\le R_A$.
Hence by replacing the $B_j$'s with unitary conjugates we can thus further assume that $R_{B_j}\le R_A$ for every $j$.  
Then all the hypotheses of the theorem are satisfied by the compressions $A_{R_A}$ and $\{(B_j)_{R_A}\}_{j=1}^{\infty}$ belonging to the factor $\mathcal{M}_{R_A}$. To simplify notations, assume henceforth that $R_A=I$. 

By using Lemma \ref{gen ess spect} ii$\,'$)  (resp., Lemma \ref{gen spect} ii$\,'$) when $\mathcal{M}$ is type III) decompose $A$ as $A=\bigoplus_{k=1}^{\infty}AE_k$ with mutually orthogonal projections  $E_k\in\{A\}'$ with $\|AE_k\|_e=\|A\|_e$ (resp., $\|AE_k\|=\|A\|$ when $\mathcal{M}$ is type III), and $AE_k$  locally invertible for all $k$. Notice that then $E_k$ must be infinite and because $R_A=I$, $R_{AE_k}=E_k$ and $\sum_{k=1}^{\infty}E_k=I$.

In case i), by using Lemma \ref{dec of sequence} decompose $B_j$ as  $B_j=\bigoplus_{k=1}^{\infty}B_jF^j_k$ for some sequence of mutually orthogonal infinite projections $\{F^j_k\}_{k=1}^{\infty}\subseteq \{B_j\}'$ with $$\sum_{k=1}^{\infty}F^j_k=I, 0\in \sigma_e\big((B_j)_{F^j_k}\big) \text{ (resp., $0\in \sigma\big((B_j)_{F^j_k}\big)$ if $\mathcal{M}$ is type III), }$$ and for each $k$, $P\prec \chi_{(\beta,\infty)}(B_jF^j_k)$ for infinitely many indices $j$. Furthermore, $\sup_{j\ge1}\|B_jF^j_k\|_e\le \alpha < \|AE_k\|_e$ for all $k$ (resp., $\sup_{j\ge1}\|B_jF^j_k\|\le \alpha < \|AE_k\|$ for all $k$ when $\mathcal M$ is type III).

In case ii), and again by Lemma  \ref{dec of sequence}, we obtain the same conclusion but with $\sup_{j\ge1}\|B_jF^j_k\|_e< \alpha \le \|AE_k\|_e$ for all $k$ (resp., $\sup_{j\ge1}\|B_jF^j_k\|< \alpha \le \|AE_k\|$ for all $k$ when $\mathcal M$ is type III).

In both cases, for every $j$, $\sum_{k=1}^\infty E_k=I$ and $\sum_{k=1}^\infty F^j_k=I$ and all the projections $E_k$ and $F^j_k$ are infinite and hence equivalent. Therefore  there is a unitary $U_j\in \mathcal M$
such that $U_jF^j_kU_j^*=E_k$ for all $k$.

Let $B'_j=U_jB_jU_j^*$, hence $B'_j=\bigoplus_{k=1}^\infty B_j'E_k$. From the above constructions we see that for every $k$, $\mathcal{M}_{E_k}$,  $A_{E_k}$ and $\{(B_j')_{E_k}\}_{j=1}^{\infty}$ satisfy the hypotheses of Proposition \ref{unit equiv, invertible}. 
Thus $AE_k=\sum_{j=1}^{\infty}C_k^j$ for some $C_k^j\in \mathcal{M}_{E_k}$ and congruent to $B_j'E_k$ {\it within} $\mathcal{M}_{E_k}$, that is such that there are partial isometries $V_k^j\in \mathcal{M}$ with $V_k^j(V_k^j)^*=(V_k^j)^*V_k^j=E_k$ for which
$C_k^j= V_k^jB_j'E_k(V_k^j)^*$.
Then $V^j:=\bigoplus_{k=1}^\infty V_k^j$ is a unitary in $\mathcal{M}$ and
\begin{align*}
A&=\bigoplus_{k=1}^{\infty}AE_k
=\bigoplus_{k=1}^{\infty}\sum_{j=1}^{\infty}V_k^jB_j'E_k(V_k^j)^*\\
&=\sum_{j=1}^{\infty}\bigoplus_{k=1}^{\infty}V_k^jB_j'(V_k^j)^*=\sum_{j=1}^{\infty}V^jB_j'(V^j)^*. 
\end{align*}

Taking $C_j:=V^jB_j'(V^j)^*\cong B_j'\cong B_j$ concludes the proof.   
\end{proof}

\section{Equivalent and unitarily equivalent copies of a single operator}

In this section we apply the results of section 4 to the case when all the operators $B_j$ are equivalent or are unitarily equivalent to a given non-zero operator $B$. We are thus able to answer affirmatively the conjecture posed by Bourin and Lee in \cite{SMV} and \cite{UE}.
We start by considering necessary conditions:

\begin{proposition}\label{nec equiv}
Let $A$, $B$ $\in \mathcal{M}^+$.
If $A=\sum_{j=1}^{\infty}C_j$ for some $C_j\sim B$ for all $j$,  then the following conditions hold:
\item [ i)] 
$\|A\|\ge \|B\|$ and $R_B\prec R_A$. If $\mathcal{M}$ is semifinite, then $\|A\|_e\ge \|B\|$.
\item [ ii)]  One of the following mutually exclusive conditions holds: 
\begin{enumerate}
\item [ 1)] $\|A\|>\|B\|$,
\item [ 2)] $\|A\|=\|B\|$ and $\chi_{\{\|B\|\}}(B)=0$,
\item [ 3)]$\|A\|=\|B\|$,
$\chi_{\{\|B\|\}}(B)\ne 0$,  $B\ne \|B\|\chi_{\{\|B\|\}}(B)$ and $\chi_{\{\|A\|\}}(A)$ is infinite,
\item [ 4)]$\|A\|=\|B\|$, $B=\|B\|\chi_{\{\|B\|\}}(B)$, $A=\|A\|\chi_{\{\|A\|\}}(A)$ and $\chi_{\{\|A\|\}}(A)$ is infinite.
\end{enumerate}
\noindent If $A=\sum_{j=1}^{\infty}C_j$ for some $C_j\cong B$ for all $j$, then in addition to the conditions i) and ii), the following conditions hold:
\item [ iii)] $N_A\prec N_B$,
\item [ iv)] $0\in \sigma(B)$ and if $\mathcal{M}$ is semifinite, then $0\in \sigma_e(B)$, 
\item [ v)] If $\|A\|=\|B\|$, then $\chi_{\{\|B\|\}}(B) \prec N_B$. 

\end{proposition}

\begin{proof}
\item [ i)] Since  $A\ge C_j\sim B$ for all $j$, we have both $\|A\|\geq \|C_j\|=\|B\|$ and $R_A\ge R_{C_j}\sim R_B$. Assume now that  $\mathcal{M}$ is semifinite.  For every $0<t<\|B\|$ we have $B\geq tR$ for some non-zero projection $R$. Hence $C_j\geq tR_j$ for some projection $R_j\sim R$ and $A\geq t\sum_{j=1}^{\infty}R_j$. Thus $\|A\|_e\geq t\left\|\sum_{j=1}^{\infty}R_j\right\|_e$ and it is enough to show that $\left\|\sum_{j=1}^{\infty}R_j\right\|_e\geq 1$. Let $T:=\sum_{j=1}^{\infty}R_j$.
Since $\tau(T)=\infty$ we see that $R_T$ is infinite.  

From \cite[Theorem 3.3]{SSP} we know that $\tau((T-I)_+)\geq \tau((R_T-T)_+)$. If $\tau((T-I)_+)=\infty$ then $\|T\|_e\geq 1$. If $\tau((T-I)_+)<\infty$ then both $(T-I)_+$ and $(R_T-T)_+$ belong to the ideal of compact operators relative to $\mathcal{M}$. Since  $$T=(T-I)_+-(R_T-T)_++R_T,$$ it follows that $\|T\|_e=\|R_T\|_e=1$.
\item [ ii)]
Let $P=\chi_{\{\|B\|\}}(B)$ and $B'= \chi_{[0, \|B\|)}(B)$ so that $B=B'\oplus \|B\|P$. Let $V_j$ be  partial isometries with $V_j^*V_j=R_B$ such that $C_j=V_jBV_j^*$ and hence $$C_j= V_jB'V_j^*\oplus \|B\|V_jPV_j^*.$$ Therefore $A=A'+A''$ where $A'=\sum_{j=1}^{\infty}V_jB'V_j^*$ and $A''=\|B\|\sum_{j=1}^{\infty}V_jPV_j^*$. 

Assume that   $P\ne 0$, then  $\|A''\|\ge \|B\|$. If furthermore $\|A\|=\|B\|$, we have $\|B\|= \|A\|\ge \|A''\|\ge \|B\|$ whence $\|\sum_{j=1}^{\infty}V_jPV_j^*\|=1$. This implies that the projections $V_jPV_j^*$ must be mutually orthogonal and hence $Q:=\sum_{j=1}^{\infty}V_jPV_j^*$ is a projection, necessarily infinite. Then $A=A'+\|B\|Q$ with $A'\geq 0$ and  $\|A\|=\|B\|$. Hence  it follows that $A'\perp Q$. Thus $Q\leq \chi_{\{\|B\|\}}(A)=\chi_{\{\|A\|\}}(A)$, whence  $\chi_{\{\|A\|\}}(A)$ is infinite.

If in addition we assume that $B=\|B\|\chi_{\{\|B\|\}}(B)$, i.e., $B'=0$, then $A'=0$, and  $A=A''=\|B\|Q= \|A\|\chi_{\{\|A\|\}}(A)$.

Now consider the case when $C_j\cong B$ for all $j$.
\item [ iii)] Obvious since then $N_B\sim N_{C_1}\ge N_A$.
\item [ iv)]
Assume first that $\mathcal{M}$ is semifinite and let $\pi:
\mathcal M\to \mathcal  M/\mathcal J$ be the quotient map. Then $\pi(A)\geq \sum_{j=1}^{N}\pi(C_j)$ for every $N\geq 1$. Assume by contradiction that $0\notin \sigma_e(B)$. Then $\pi(B)$ is invertible in $\mathcal{M}/\mathcal{J}$, hence $\pi(B)\geq t\pi(I)$ for some positive $t$. Since $\pi(C_j)\cong \pi(B)$ in $\mathcal{M}/\mathcal{J}$,  we have $\pi(A)\geq Nt\pi(I)$ for every $N$,   a contradiction.
Thus $0\in \sigma_e(B)$ and hence $0\in \sigma(B)$. When $\mathcal{M}$ is type III,  the same argument, without the need to pass to the quotient  algebra $\mathcal{M}/\mathcal{J}$, shows that $0\in \sigma(B)$.
\item [ v)]  If $\chi_{\{\|B\|\}}(B)\ne 0$ and $\xi \in \chi_{\|B\|}(C_i)\mathcal{H}$ is a unit vector, then $ \langle C_i\xi,\xi\rangle= \|B\|$, hence
$$\|B\|= \|A\|\ge  \langle A\xi,\xi\rangle=\|B\|+\sum_{j\ne i}^{\infty}\langle C_j\xi,\xi\rangle$$ 
which implies that $\xi\in N_{C_j}\mathcal{H}$ for every $j\ne i$, i.e.,  $\chi_{\{\|C_i\|\}}(B)\le N_{C_j}$.  But since $C_j\cong B$ implies that $N_{C_j}\sim N_B$, it follows that $\chi_{\{\|B\|\}}(B)\prec N_B$ .

\end{proof}



Next we present sufficient conditions for the decomposition of $A$ into sums of positive operators equivalent to a fixed positive operator $B$. These are of course based on the decompositions obtained in section 3, but with the two additional cases iii) and iv).

\begin{theorem}\label{equiv}
Let $\mathcal{M}$ be a $\sigma$-finite, semifinite infinite factor, $A$, $B\in \mathcal{M}^+$, and $B\ne0$. Each of the following conditions is sufficient for $A=\sum_{j=1}^{\infty}C_j$ for some $C_j\sim B$ with $N_{C_j}$ infinite for all $j$:
\item [ i)] $\tau((A-\|B\|I)_+)=\infty$,
\item [ ii)] $\|A\|_e=\|B\|$ and $\chi_{\{\|B\|\}}(B)=0$,
\item [ iii)] $\|A\|_e=\|B\|$,
$\chi_{\{\|B\|\}}(B)\ne 0$, $B\ne \|B\|\chi_{\{\|B\|\}}(B)$ and $\chi_{\{\|A\|_e\}}(A)$ is infinite,
\item [ iv)] $\|A\|_e=\|B\|$, $B=\|B\|\chi_{\{\|B\|\}}(B)$, $A=\|A\|_e\chi_{\{\|A\|_e\}}(A)$ and $\chi_{\{\|A\|_e\}}(A)$ is infinite.
\end{theorem}
\begin{proof}
 i) and ii) are direct consequences of Theorem \ref{equiv infinite trace}.

\item [ iii)]\& iv)  Let $P:=\chi_{\{\|B\|\}}(B)\ne 0$. By decomposing $\chi_{\{\|A\|_e\}}(A)$ into  the sum of two infinite projections, we can assume that $A=A'\oplus \|B\|Q$ where $\|A'\|_e=\|B\|$ and $Q$ is an infinite projection. Decompose further   $Q=\sum_{j=1}^{\infty}Q_j$ into the sum of projections $Q_j\sim P$. Now decompose $B$ as $B=B'\oplus \|B\|P$ where $B'=B\chi_{(0, \|B\|)}(B)$. 
\\
Assume that iii) holds, i.e.,  $B'\ne 0$. We can decompose $A'=\sum_{j=1}^{\infty}C'_j$ with $C'_j\sim B'$ by invoking i) in the case  that $\|B'\|<\|B\|$,  and hence $\|A'\|_e>\|B'\|$, whence $\tau((A'-\|B'\|)_+)=\infty$ or invoking ii) in the case that $\|B'\|=\|B\|$ because then $\chi_{\{\|B'\|\}}(B')=0$. 
Thus
$$A= \Big(\sum_{j=1}^{\infty}C'_j\Big)\oplus  \Big(\sum_{j=1}^{\infty}\|B\|Q_j\Big)=\sum_{j=1}^{\infty}\big(C'_j\oplus \|B\|Q_j\big)$$
and $C_j:=C'_j\oplus \|B\|Q_j\sim B'\oplus \|B\|P= B$.
Notice that by construction, $N_{C_j}\ge \sum_{k\ne j}Q_k$ is infinite for every $j$.

Finally, notice that  $B'=A'=0$ in case iv) and hence the proof is a special case of iii).
\end{proof}

Thus combining these sufficient conditions with the   necessary conditions of Proposition   \ref {nec equiv} we obtain:
\begin{corollary} \label{norm=essnorm}
Let $\mathcal M$ be semifinite, $A, B\in \mathcal M^+$, $B\ne 0$ and $\|A\|_e=\|A\|$, then the condition ii) in Proposition \ref {nec equiv} is necessary and sufficient for $A=\sum_{j=1}^{\infty}C_j$ for some $C_j\sim B$. Furthermore if that condition is satisfied, the decomposition can be chosen so that $N_{C_j}$ is infinite for all $j$.
\end{corollary}
It is easy to see that condition ii) in Proposition \ref {nec equiv} is also sufficient in the type III case:
\begin{corollary}\label{N&S type III}
If $\mathcal M$ is type III, $A, B\in \mathcal M^+$ with  $B\ne 0$, then the condition ii) in Proposition \ref {nec equiv} is necessary and sufficient for $A=\sum_{j=1}^{\infty}C_j$ for some $C_j\sim B$. Furthermore if that condition is satisfied, the decomposition can be chosen so that $N_{C_j}\ne 0$ for all $j$.
\end{corollary}
 
Decompositions into sums of operators  unitarily equivalent to a given positive operator are a special case of Theorem \ref {unit equiv sequence}, but here too we can add the two additional conditions iii) and iv). These cases are easily obtained from the fact that equivalent positive operators with infinite null spaces are unitarily equivalent.

\begin{corollary}\label{UE}
Let $\mathcal{M}$ be  semifinite, $A$, $B\in \mathcal{M}^+$, $B\neq 0$, $0\in \sigma_e(B)$, and that $N_A\prec N_B$. Then any of the following mutually exclusive conditions implies that $A=\sum_{j=1}^{\infty}C_j$ for some $C_j\cong B$:
\item [ i)] $\|A\|_e>\|B\|$,
\item [ ii)] $\|A\|_e=\|B\|$ and $\chi_{\{\|B\|\}}(B)=0$, 
\item [ iii)] $\|A\|_e=\|B\|$,  $\chi_{\{\|B\|\}}(B)\ne 0$, $B\ne \|B\|\chi_{\{\|B\|\}}(B)$, $\chi_{\{\|A\|_e\}}(A)$ is infinite and $N_B$ is infinite,
\item [ iv)] $\|A\|_e=\|B\|$, $B=\|B\|\chi_{\{\|B\|\}}(B)$, $A=\|A\|\chi_{\{\|A\|\}}(A)$, $\chi_{\{\|A\|\}}(A)$ is infinite and $N_B$ is infinite.
\end{corollary}

\begin{corollary}\label{P:NecSuf}
Let $\mathcal{M}$ be  type III, $A$, $B\in \mathcal{M}^+$, $B\ne 0$.
Then $A=\sum_{j=1}^{\infty}C_j$ with $C_j\cong B$ if and only if 
one of the following mutually exclusive conditions holds:
\item [i)] $\|A\|>\|B\|$, $0\in \sigma(B)$ and $N_A\prec N_B$,
\item [ii)] $\|A\|=\|B\|$, $\chi_{\{\|B\|\}}(B)=0$, $0\in \sigma(B)$ and $N_A\prec N_B$,
\item [iii)]  $\|A\|=\|B\|$,
$\chi_{\{\|B\|\}}(B)\ne 0$, $B\ne \|B\|\chi_{\{\|B\|\}}(B)$, $\chi_{\{\|A\|\}}(A)\ne 0$ and $N_B\ne 0$,
\item [iv)]  $\|A\|=\|B\|$, $B=\|B\|\chi_{\{\|B\|\}}(B)$, $A=\|A\|\chi_{\{\|A\|\}}(A)$, $\chi_{\{\|A\|\}}(A)\ne 0$ and $N_B\ne 0$.
\end{corollary}

\end{document}